\documentclass[10pt,a4paper]{amsart}
\usepackage{amsmath}
\usepackage{amsfonts}
\usepackage{color}

\def\cqfd{
{\hfill
\kern 6pt\penalty 500
\raise -1pt\hbox{\vrule\vbox to 5pt{\hrule width 4pt
\vfill\hrule}\vrule}}
\break}

\long\def\eg#1{{\color{green}#1}}
\long\def\eg#1{}

\newtheorem{theorem}{Theorem}
\newtheorem{proposition}[theorem]{Proposition}

\newtheorem{corollary}[theorem]{Corollary}

\begin{document}

\date{\today}

\title[Jacobians  surfaces]{The minimum and maximum number of rational points on Jacobian surfaces over finite fields}
\author{Safia Haloui}
\address{Institut de Math\'ematiques de Luminy, Marseille, France}

\email{haloui@iml.univ-mrs.fr}



\begin{abstract}
We give some bounds on the numbers of rational points on abelian varieties and jacobians  varieties over finite fields. The main result is that we determine the maximum and minimum number of rational points on jacobians varieties of dimension 2.
\end{abstract}

\maketitle




\section{Introduction}

We are interested in the number of rational points over finite felds on particular abelian varieties, namely jacobian varieties. 
In particular, for given integers $q=p^e$ ($p$ prime and $e\geq 1$) and $g\geq 1$, we consider the quantities
$$J_q(g)=\max_{X}\# J_X(\mathbb{F}_q)\quad \mbox{ and }\quad j_q(g)=\min_{X}\# J_X(\mathbb{F}_q)$$
where $X$ runs over the set of (absolutely irreducible, smooth, projective, algebraic) curves of genus $g$ defined over the finite field $\mathbb{F}_q$ and $J_X$ is the jacobian of $X$.

After giving some general bounds for these quantities, we determine the values of $J_q(2)$ and $j_q(2)$.

\bigskip

Let $A$ be an abelian variety defined over $\mathbb{F}_q$ of  dimension $g$. The characteristic polynomial $f_A(t)$ of $A$ is defined to be the characteristic polynomial of its Frobenius endomorphism. This is a polynomial with integer coefficients and its set of roots has the form 
$$\{\omega_1,\overline{\omega_1},\dots ,\omega_g, \overline{\omega_g}\}$$
 with $\vert\omega_i\vert=\sqrt{q}$ (Riemann Hypthesis for abelian varieties proved by Weil). For  $1\leq i\leq g$, we set  $x_i=-(\omega_i+\overline{\omega_i})$ and we refer  the sum $\sum_{i=1}^g(\omega_i+\overline{\omega_i})=-\sum_{i=1}^gx_i$ as the trace of $A$. 
 
The data of the $x_i$'s is equivalent to the data of the isogeny class of $A$ (by the Honda-Tate Theorem); we will say that $A$ has type $[x_1,\dots ,x_g]$. The type of an absolutely irreducible, smooth, projective, algebraic curve over $\mathbb{F}_q$ (as defined in \cite{lau1}) is the type of its jacobian.

The number of rational points on $A$ is given by
$$\#A(\mathbb{F}_q)=f_A(1),$$
and since 
$$f_A(t)=\prod_{i=1}^g(t-\omega_i)(t-\overline{\omega_i})=\prod_{i=1}^g(t^2+x_it+q),$$
we obtain
\begin{eqnarray}
\# A(\mathbb{F}_q)=\prod_{i=1}^{g}(q+1+x_i). \label{card}
\end{eqnarray}




\section{Abelian varieties}

Let $A$ be an abelian variety over $\mathbb{F}_q$ of  dimension $g$ and of type $[x_1,\dots ,x_g]$. Since each $x_i$ has an absolute value less than or equal to $2\sqrt{q}$ by the Riemann Hypothesis, we obtain from (\ref{card}) the well-known  Weil bounds  for abelian varieties:
$$(q+1-2\sqrt{q})^g\leq \# A(\mathbb{F}_q)\leq(q+1+2\sqrt{q})^g.$$ 

We can derive an upper bound for $ \# A(\mathbb{F}_q)$ which depends on the trace of $A$. Indeed, recall that, if $c_1,\ldots, c_n$ are non negative real numbers,  the arithmetic-geometric mean states that
$$\sqrt[n]{c_1\ldots c_n}\leq\frac{1}{n}(c_1+\cdots +c_n)$$
with equality if and only if the $c_i$'s are equal.
Applying this inequality to  (\ref{card}), we get the following proposition, independently shown by Quebbemann \cite{queb} and Perret \cite{per}:

\begin{proposition}[Perret-Quebbemann]\label{bornetr}We have
$$\# A(\mathbb{F}_q)\leq \Bigl(q+1+\frac{\sum_{i=1}^g x_i}{g}\Bigr)^g$$
with equality if and only if the $x_i$'s are equal.
\end{proposition}

We set $m=[2\sqrt{q}]$. Using the arithmetic-geometric mean, Serre  \cite{serre0} proved that
\begin{eqnarray}
\sum_{i=1}^g x_i\leq gm, \label{trace}
\end{eqnarray}
and so Proposition \ref{bornetr} implies that
$$\# A(\mathbb{F}_q)\leq(q+1+m)^g.$$

It is natural to ask whether $\# A(\mathbb{F}_q)$ has a lower bound of this kind, and the answer turns out to be yes.

\begin{theorem}\label{SerreWeil}
Let $A$ be an abelian variety defined over $\mathbb{F}_q$ of  dimension $g$. Then
$$(q+1-m)^g\leq\# A(\mathbb{F}_q)\leq(q+1+m)^g.$$
The first inequality is an equality if and only if the $x_i$'s are equal to $-m$ and the second if and only if they are equal to $m$.
\end{theorem}
\begin{proof}Only the first inequality has to be proved (note the similarity of the following proof with the one of (\ref{trace})).

For $k=0,\dots ,g $ let $t_k$ be the $k$th symmetric function of the $(m+1+x_i)$'s, for $i=1,\dots ,g$ (i.e. $\prod_{i=1}^g(t+(m+1+x_{i}))=\sum_{k=0}^gt_kt^{g-k}$).

Let $k\in\{ 1,\dots ,g\}$ and define the quantity
$$T_k=\prod_{\{ i_1,\dots ,i_k\}\subseteq\{ 1,\dots ,g\} }\prod_{j=1}^k(m+1+x_{i_j}).$$
$T_k$ is a non zero integer (since it is an algebraic integer invariant by $\mathbb{Q}$-automorphisms  of $\mathbb{C}$). Thus
\begin{eqnarray}
T_k\geq1. \label{1}
\end{eqnarray}
On the other hand, using the arithmetic-geometric mean, we obtain
\begin{eqnarray}
T_k^{1/\binom{g}{k}}\leq\frac{1}{\binom{g}{k}}\sum_{\{ i_1,\dots ,i_k\}\subseteq\{ 1,\dots ,g\} }\prod_{j=1}^k(m+1+x_{i_j})=\frac{1}{\binom{g}{k}}t_k. \label{2}
\end{eqnarray}
Combining (\ref{1}) and (\ref{2}), we find
\begin{eqnarray}
\binom{g}{k}\leq t_k. \label{3}
\end{eqnarray}
Moreover, (\ref{3}) remains true for $k=0$.

Multiplying  both sides of (\ref{3}) by $(q-m)^{g-k}$ and adding the inequalities obtained for $k=0,\dots ,g $ we obtain
$$\sum_{k=0}^g\binom{g}{k}(q-m)^{g-k}\leq \sum_{k=0}^gt_k(q-m)^{g-k} $$
from which the result follows.
\end{proof}

 
 
\section{Jacobians}
Since jacobians are particular cases of abelian varieties, the bounds of the previous section can be applied. Using Theorem \ref{SerreWeil}, we have:
\begin{corollary} \label{WeilJacobian}
$$(q+1-m)^g\leq j_q(g)\leq J_q(g)\leq(q+1+m)^g.$$
\end{corollary}

Since the number of rational points on a curve $X$ of genus $g$ and of type $[x_1,\ldots,x_g]$ is related to the trace of its Frobenius by

$$\# X({\mathbb F}_{q})=q+1+\sum_{i=1}^gx_i,$$
Proposition \ref{bornetr} gives us
\begin{eqnarray}
\# J_X(\mathbb{F}_q)\leq \Bigl(q+1+\frac{\# X(\mathbb{F}_q)-(q+1)}{g}\Bigr)^g \label{pirrete}
\end{eqnarray}
and we find again the bound proved by Quebbemann \cite{queb} and Perret \cite{per}. These bounds should be compared with those of 
Lachaud and Martin-Deschamps \cite{L-MD}:

$$(\sqrt{q}-1)^2\frac{q^{g-1}-1}{g}\frac{\# X({\mathbb F}_q)+q-1}{q-1}\leq \# J_X({\mathbb F}_q).$$
Moreover if $X$ admits a map of degree $d$ onto ${\mathbb P}^1$, we have

$$\# J_X({\mathbb F}_q)\leq e(2g\sqrt{e})^{d-1}q^g$$
where $e$ is the base of the natural logarithm.

The bound (\ref{pirrete}) is better than the Lachaud-Martin-Deschamps bound when $g$ is large and $q$ is fixed, except when $X$ has a small gonality with respect to the genus.

\bigskip

Let  $N_q(g)=\max_X\# X({\mathbb F}_q)$ where $X$ runs over the set of (absolutely irreducible, smooth, projective, algebraic) curves of genus $g$ defined over  $\mathbb{F}_q$. Then the bound (\ref{pirrete}) can be applied to give the following proposition.

\begin{proposition}\label{JqgetNqg}We have
$$J_q(g)\leq \Bigl(q+1+\frac{N_q(g)-(q+1)}{g}\Bigr)^g$$
with equality if and only if there exists an optimal curve (i.e. with $N_q(g)$ points) with only one Frobenius eigenvalue (with its conjugate).
\end{proposition}

\bigskip
\noindent
{\bf Remarks.}

\medskip
\noindent
{\bf 1.} A curve that reaches the Serre-Weil bound, i.e. such that
$$\# X({\mathbb F}_q)=q+1+m$$
must  have $x_i= m$ for all $i$ (see \cite{lau2}). Thus, by Corollary \ref{WeilJacobian},  its jacobian is optimal i.e. has $J_q(g)$ points.

\medskip
\noindent
{\bf 2.} When $g=2$, a jacobian with $J_q(g)$ points is the jacobian of an optimal curve (i.e. with $N_q(g)$ points) but the jacobian of an optimal curve doesn't always have $J_q(g)$ points. For example, when $q=3$, there exist genus $2$ curves with $N_3(2)=8$ points whose jacobian have $36$, $35$, $34$ and $33$ points (it follows from the results of \cite{hnr}).

\vskip1cm

The case of jacobians of dimension 1 corresponds to elliptic curves. The values of $J_q(1)$ and $j_q(1)$ where calculated by Waterhouse \cite{water}.

Recall that $q=p^e$.
\begin{proposition}\label{J_q(1)}
\noindent
\begin{enumerate}
\item $J_q(1)$ is equal to
\begin{itemize}
\item $(q+1+m)$ if $e=1$, $e$ is even or $p\not\vert m$
\item $(q+m)$ otherwise.
\end{itemize}
\item $j_q(1)$ is equal to
\begin{itemize}
\item $(q+1-m)$ if $e=1$, $e$ is even or $p\not\vert m$
\item $(q+2-m)$ otherwise.
\end{itemize}
\end{enumerate}
\end{proposition}




\section{Abelian surfaces over finite fields}
In order to determine $J_q(2)$, we focus on abelian varieties of dimension 2.
Let $A$ be an abelian surface over $ \mathbb{F}_q$ of type $[x_1,x_2]$.

Its characteristic polynomial $f_A(t)$ has the form
$$f_A(t)=t^4+a_1t^3+a_2t^2+qa_1t+q^2$$
with $a_1$ and $a_2$ integers.

We consider the real Weil polynomial of $A$:
$${\tilde f}_A(t)=(t+x_1)(t+x_2)=t^2+a_1t+a_2-2q.$$

By elementary computations, we can show that (see \cite{ruck} or \cite{mana}) the fact that the roots of $f_A(t)$ are $q$-Weil numbers (i.e. algebraic integers such that their images under every complex embedding have absolute value $\sqrt{q}$) is equivalent to

\begin{eqnarray}
\vert a_1\vert\leq 2m\quad \mbox{ and }\quad 2\vert a_1\vert\sqrt{q}-2q\leq a_2\leq\frac{a_1^2}{4}+2q. \label{a12}
\end{eqnarray}

Moreover,
\begin{eqnarray}
\# A(\mathbb{F}_q)=q^2+1+(q+1)a_1+a_2.\label{abeliansurface}
\end{eqnarray}
Table  \ref{tabmax} gives all the possibilities for $(a_1,a_2)$ such that $a_1\geq 2m-2$.

\begin{table}[htbp]
\begin{center}
\begin{tabular}{| l l | l | l |}
\hline
$a_1$ & $a_2$ & Type & Nb of pts \\
\hline
$2m$ & $m^2+2q$ & $[m,m]$ & $(q+1+m)^2$ \\
\hline
$2m-1$ & $m^2-m+2q$ & $[m,m-1]$ & $(q+1+m)(q+m)$ \\
 & $m^2-m-1+2q$ & $[m+\frac{-1+\sqrt{5}}{2},m+\frac{-1-\sqrt{5}}{2}]$ & $(q+1+m+\frac{-1+\sqrt{5}}{2})(q+1+m+\frac{-1-\sqrt{5}}{2})$ \\
\hline 
$2m-2$ & $m^2-2m+1+2q$ & $[m-1,m-1]$ & $(q+m)^2$ \\
 & $m^2-2m+2q$ & $[m,m-2]$ & $(q+1+m)(q-1+m)$ \\
 & $m^2-2m-1+2q$ & $[m-1+\sqrt{2},m-1-\sqrt{2}]$ & $(q+m+\sqrt{2})(q+m-\sqrt{2})$ \\
 & $m^2-2m-2+2q$ & $[m-1+\sqrt{3},m-1-\sqrt{3}]$ & $(q+m+\sqrt{3})(q+m-\sqrt{3})$ \\
\hline
\end{tabular}
\end{center}
\caption{Couples $(a_1,a_2)$ maximizing the number of points on $A$}
\label{tabmax}
\end{table}

The numbers of points are classified in decreasing order and an abelian variety with $(a_1,a_2)$ not in the table has a number of points strictely less than the values of the table. Indeed, if $-2m\leq a_1< 2m-2$, we have

\begin{eqnarray*}
(q+1)a_1+a_2 & \leq & [(q+1)a_1+\frac{{a_1}^2}{4}+2q]\\
& \leq & [(q+1)(2m-3)+\frac{(2m-3)^2}{4}+2q]\\
& = & (q+1)(2m-3)+m^2-3m+2+2q\\
& = & (q+1)(2m-2)+(m^2-2m-2+2q)+(3-(q+m))\\
& < & (q+1)(2m-2)+(m^2-2m-2+2q).
\end{eqnarray*}
For the second inequality, note that the function $x\mapsto (q+1)x+\frac{{x}^2}{4}$ is increasing on the interval $[-2m,2m-3]$.

\medskip

In the same way, we build the table of couples $(a_1,a_2)$ with $a_1\leq -2m+2$ (note that the extremities of the interval containing $a_2$ given by (\ref{a12}) depend only on the value of $a_1$, hence the $a_2$'s are the same as in the previous table).

\begin{table}[htbp]
\begin{center}
\begin{tabular}{| l l | l | l |}
\hline
$a_1$ & $a_2$ & Type & Nb of pts \\
\hline
$-2m$ & $m^2+2q$ & $[-m,-m]$ & $(q+1-m)^2$ \\
\hline
$-2m+1$ & $m^2-m-1+2q$ & $[-m+\frac{1+\sqrt{5}}{2},-m+\frac{1-\sqrt{5}}{2}]$ & $(q+1-m+\frac{1+\sqrt{5}}{2})(q+1-m+\frac{1-\sqrt{5}}{2})$ \\
 & $m^2-m+2q$ & $[-m,-m+1]$ & $(q+1-m)(q+2-m)$ \\
\hline 
$-2m+2$ & $m^2-2m-2+2q$ & $[-m+1+\sqrt{3},-m+1-\sqrt{3}]$ & $(q+2-m+\sqrt{3})(q+2-m-\sqrt{3})$ \\
 & $m^2-2m-1+2q$ & $[-m+1+\sqrt{2},-m+1-\sqrt{2}]$ & $(q+2-m+\sqrt{2})(q+2-m-\sqrt{2})$ \\
 & $m^2-2m+2q$ & $[-m,-m+2]$ & $(q+1-m)(q+3-m)$ \\
 & $m^2-2m+1+2q$ & $[-m+1,-m+1]$ & $(q+2-m)^2$ \\
\hline
\end{tabular}
\end{center}
\caption{Couples $(a_1,a_2)$ minimizing the number of points on $A$}
\label{tabmin}
\end{table}

Again, the numbers of points are classified in increasing order and an abelian variety with $(a_1,a_2)$ not in the following table has a number of points strictly greater than the values of the table. Indeed, if $-2m+2< a_1\leq 2m$, we have

\begin{eqnarray*}
(q+1)a_1+a_2 & \geq & (q+1)a_1+2\vert a_1\vert\sqrt{q}-2q\\
& \geq &  (q+1)(-2m+3)+2(2m-3)\sqrt{q}-2q\\
& = &  (q+1)(-2m+2)+2(2m-3)\sqrt{q}-q+1\\
& = & (q+1)(-2m+2)+(m^2-2m+1+2q)-(m^2-2m+1+2q)+2(2m-3)\sqrt{q}-q+1\\
& = & (q+1)(-2m+2)+(m^2-2m+1+2q)-(m-1)^2+4(m-1)\sqrt{q}-4q+q+2\sqrt{q}+1\\
& = & (q+1)(-2m+2)+(m^2-2m+1+2q)-(m-1-2\sqrt{q})^2+(\sqrt{q}+1)^2\\
& > & (q+1)(-2m+2)+(m^2-2m+1+2q).
\end{eqnarray*}
For the second inequality, note that the function $x\mapsto (q+1)x+2\vert x\vert\sqrt{q}$ is increasing on $[-2m+3,2m]$ .




\section{Jacobians of dimension $2$}
 
In this section, we determine $J_q(2)$ and $ j_q(2)$.

\medskip
Theorems \ref{Jq2} and \ref{jq2} will be proved in the following way:

\begin{enumerate}
\item We look at the highest row of Table \ref{tabmax} or \ref{tabmin} (depending on the theorem being proved).

\item Then we check if the corresponding polynomial is the characteristic polynomial of an abelian variety.

\item When it is the case, we check if this abelian variety is isogenous to a jacobian  variety.

\item 
When it is not the case, we look at the following row and we come back to the second step.

\end{enumerate}

For the second step, we use the results of R\"uck \cite{ruck} and Maisner-Nart-Howe \cite{mana} who solved the problem of describing characteristic polynomials of abelian surfaces, in particular the fact that if $(a_1,a_2)$ satisfy (\ref{a12}) and $p$ does not divide $a_2$ then the corresponding polynomial is the characteristic polynomial of an abelian surface. 

For the third step, we use \cite{hnr} where we can find a characterization of isogeny classes of abelian surfaces containing a jacobian.

\medskip

The determination of $J_q(2)$ in Theorem \ref{Jq2} is closely related to that of $N_q(2)$ as was done by Serre \cite{serre2}.

In order to simplify the proof of Theorem \ref{jq2}, we use the fact that given a curve of genus $2$, if we denote by $(a_1,a_2)$ the coefficients of its characteristic polynomial, there exists a curve (its quadratic twist) whose coefficients are $(-a_1,a_2)$. This allows us to adapt the proof of Theorem \ref{Jq2}.

\medskip

Let us recall the definition of special numbers introduced by Serre. An odd power $q$ of a prime number $p$ is \textit{special} if one of the following conditions is satisfied (recall that $m=[2\sqrt{q}]$):

\begin{enumerate}
\item $m$ is divisible by $p$,
\item there exists $x\in\mathbb{Z}$ such that $q=x^2+1$,
\item there exists $x\in\mathbb{Z}$ such that $q=x^2+x+1$,
\item there exists $x\in\mathbb{Z}$ such that $q=x^2+x+2$.
\end{enumerate}

In \cite{serre}, Serre asserts that if $q$ is prime then the only possible conditions are conditions (2) and (3). When $q$ is not a prime then condition (2) is impossible,  condition (3) is possible only if $q=7^3$ and condition (4) is possible only if $q=2^3$, $2^5$ or $2^{13}$. Moreover, using basic arithmetic, it can be shown (see \cite{lau1} for more details) that conditions (2), (3) and (4) are respectively equivalent to $m^2-4q=-4$, $-3$ and $-7$.

\bigskip
The complete set of possible values of $J_q(2)$ is given in the following theorem.

\begin{theorem}\label{Jq2}
\noindent
\begin{itemize}
\item[a)] If $q$ is a square, then  $J_q(2)$ equals:
\begin{itemize}
\item[$\bullet$] $(q+1+m)^2$ if $q\neq 4,9$
\item[$\bullet$] $55$ if $q=4$
\item[$\bullet$] $225$ if $q=9$
\end{itemize}
\noindent
\item[b)] If $q$ is not a square, then  $J_q(2)$ equals:
\begin{itemize}
\item[$\bullet$] $(q+1+m)^2$ if $q$ is not special
\item[$\bullet$] $(q+1+m+\frac{-1+\sqrt{5}}{2})(q+1+m+\frac{-1-\sqrt{5}}{2})$ if $q$ is special and $\{ 2\sqrt q\}\geq\frac{\sqrt 5 -1}{2}$
\item[$\bullet$] $(q+m)^2$ if $q$ is special, $\{ 2\sqrt q\}<\frac{\sqrt 5 -1}{2}$, $p\neq 2$ or $p\vert m$
\item[$\bullet$] $(q+1+m)(q-1+m)$ otherwise.
\end{itemize}
\end{itemize}
\end{theorem}

\begin{proof}

\textbf{a) If $q$ is a square.}

\noindent
$\bullet$ Let $q\neq 4,9$.  Since $N_q(2)$ is the Serre-Weil bound (see \cite{serre}), by Proposition \ref{JqgetNqg} a curve reaching this bound has type $[m,m]$ (see \cite{serre2} for an explicit construction).

\medskip
\noindent
$\bullet$ Let $q=4$. First we prove that $J_4(2)\leq 55$. We have $m=4$ and $q+1+2m=13$. Every curve $X$ of genus  $2$ over $\mathbb{F}_q$ is hyperelliptic and therefore the number of rational points is at most $2(q+1)=10=13-3$. Since $\# X({\mathbb F}_q)=q+1+a_1$, we deduce that a jacobian of dimension 2 over  $\mathbb{F}_4$ must have $a_1\leq 2m-3=5$. 

If $a_1=5$ then by (\ref{a12}) we have:  $a_2\leq 14$. The real Weil polynomial ${\tilde f}_A(t)$ of an abelian variety $A$ over ${\mathbb F}_4$ with $a_1=5$ and $a_2=14$ is
\begin{eqnarray*}
{\tilde f}_A(t)=t^2+a_1 t+a_2-2q & = & t^2+5t+6\\
& = &  (t+3)(t+2)\\
& = &  (t+(m-1))(t+(m-2)).
\end{eqnarray*}
This kind of polynomial is never the real Weil polynomial of a jacobian  (because $(m-1)-(m-2)=1$, see \cite{hnr}). Thus we have $a_2\leq 13$ and a jacobian surface over ${\mathbb F}_4$ with $a_1=5$ has at most  $q^2+1+5(q+1)+13=55$ points by relation (\ref{abeliansurface}).

If $a_1<5$, we have
\begin{eqnarray*}
q^2+1+(q+1)a_1+a_2 & \leq & q^2+1+(q+1)a_1+\frac{{a_1}^2}{4}+2q\\
& = & 25+5a_1+\frac{{a_1}^2}{4}\\
& \leq & 25+5\times 4+\frac{4^2}{4}\\
& = & 49
\end{eqnarray*}
(for the third row, note that the function $x\mapsto 5x+\frac{{x}^2}{4}$ is increasing on $[-8,4]$ (and $a_1\geq -8$)). Thus  an abelian surface over ${\mathbb F}_4$ with $a_1<5$ has less than  $55$ points, hence $J_4(2)\leq 55$.

It remains to prove that $J_4(2)\geq 55$. It is enough to prove that there exists a jacobian with $(a_1,a_2)=(5,13)$ (because such a jacobian will have $q^2+1+5(q+1)+13=55$ points). There exists an abelian surface with $(a_1,a_2)=(5,13)$ (because $p=2$ does not divide $13$). Its  real Weil  polynomial is 
\begin{eqnarray*}
{\tilde f}_A(t)=t^2+5 t+13-2q & = & t^2+5t+5\\
& = &  (t+\frac{5+\sqrt 5}{2})(t+\frac{5-\sqrt 5}{2})\\
& = &  (t+(m-\frac{3+\sqrt 5}{2}))(t+(m-\frac{3-\sqrt 5}{2})).
\end{eqnarray*}
Therefore its $x_i$'s are not integers and thus this abelian surface is simple. Finally, using  \cite{hnr}, we find that  it is isogenous to a jacobian. 

\medskip
\noindent
$\bullet$ If $q=9$, we have $m=6$ and $q+1+2m=22$. Moreover, $2(q+1)=20=22-2$. Hence, we must have $a_1\leq 2m-2$. The highest row of Table \ref{tabmax} such that $a_1 = 2m-2$ is that with type $[m-1,m-1]$. By \cite{water} there exists an elliptic curve of trace $-(m-1)$ and by \cite{hnr} the product of two copies of this curve is isogenous to a jacobian and has $(q+m)^2=225$ points.

\bigskip
\noindent
\textbf{b) If $q$ is not a square.}

In \cite{serre2}, Serre proved the following facts:

\medskip
\noindent
$\bullet$ There exists a jacobian of type $[m,m]$ if and only if $q$ is not special.

\medskip
\noindent
$\bullet$ An abelian surface of type $[m,m-1]$ is never a jacobian.

\medskip
\noindent
$\bullet$ If $q$ is special, then there exists a jacobian of type  $[m+\frac{-1+\sqrt{5}}{2},m+\frac{-1-\sqrt{5}}{2}]$ if and only if  $\{ 2\sqrt q\}\geq\frac{\sqrt 5 -1}{2}$. Note that $\{ 2\sqrt q\}\geq\frac{\sqrt 5 -1}{2}$ is equivalent to $m+\frac{-1+\sqrt{5}}{2}\leq 2\sqrt q$, thus it is obvious that this condition is necessary.

\medskip
\noindent
$\bullet$ If $q$ is special, $\{ 2\sqrt q\}<\frac{\sqrt 5 -1}{2}$, $p\neq 2$ or $p\vert m$, then there exists a jacobian of type $[m-1,m-1]$.

\medskip
\noindent
$\bullet$ If $q$ is special, $\{ 2\sqrt q\}<\frac{\sqrt 5 -1}{2}$, $p=2$ and $p\not\vert m$, that is, $q=2^5$ or $2^{13}$ (for $q=2^3$, we have $\{ 2\sqrt q\}\geq\frac{\sqrt 5 -1}{2}$), then there exists a jacobian of type $[m,m-2]$.

\medskip

It remains to prove that for $q=2^5$ and $2^{13}$, there does not exist a jacobian of type $[m-1,m-1]$. In fact, when $q=2^5$ and $2^{13}$, an abelian variety with all the $x_i$'s equal to $(m-1)$ must have a dimension respectively multiple of $5$ and $13$ (see \cite{mana}, Prop. 2.5).

\end{proof}

\bigskip

The complete set of possible values of $j_q(2)$ is given in the following theorem.

\begin{theorem}\label{jq2}
\noindent
\begin{itemize}
\item[a)] If $q$ is a square, $j_q(2)$ equals:
\begin{itemize}
\item[$\bullet$] $(q+1-m)^2$ if $q\neq 4,9$
\item[$\bullet$] $5$ if $q=4$
\item[$\bullet$] $25$ if $q=9$
\end{itemize}
\noindent
\item[b)] If $q$ is not a square, $j_q(2)$ equals:
\begin{itemize}
\item[$\bullet$] $(q+1-m)^2$ if $q$ is not special
\item[$\bullet$] $(q+1-m+\frac{1+\sqrt{5}}{2})(q+1-m+\frac{1-\sqrt{5}}{2})$ if $q$ is special and $\{ 2\sqrt q\}\geq\frac{\sqrt 5 -1}{2}$
\item[$\bullet$] $(q+2-m-\sqrt{2})(q+2-m+\sqrt{2})$ if $q$ is special and $\sqrt{2}-1\leq\{ 2\sqrt q\}<\frac{\sqrt 5 -1}{2}$
\item[$\bullet$] $(q+1-m)(q+3-m)$ if $q$ is special, $\{ 2\sqrt q\} <\sqrt 2-1 $, $p\not\vert m$ and $q\neq 7^3$
\item[$\bullet$] $(q+2-m)^2$ otherwise.
\end{itemize}
\end{itemize}
\end{theorem}

\begin{proof}

\textbf{a) If $q$ is a square.}

\noindent
$\bullet$ Using twisting arguments and the proof of Theorem \ref{Jq2}, we see that if $q\neq 4,9$, there exists a jacobian of type $[-m,-m]$. 

\medskip
\noindent
$\bullet$ Let $q=4$ (and $m=4$). First we prove that $j_4(2)\geq 5$. We have $a_1\geq -5$ since the quadratic twist of a curve with $a_1<-5$ would have $a_1>5$ which is not possible (see the proof of Theorem \ref{Jq2}).

If $a_1=-5$ then by (\ref{a12}) we have $a_2\geq 12$. The real Weil polynomial of an abelian variety $A$ over ${\mathbb F}_4$ with $a_1=-5$ and $a_2=12$ is
\begin{eqnarray*}
{\tilde f}_A(t)=t^2+a_1 t+a_2-2q & = & t^2-5t+4\\
& = &  (t-4)(t-1)\\
& = &  (t-m)(t-(m-3)).
\end{eqnarray*}
This is not the real Weil polynomial of a jacobian (it is one of an almost ordinary abelian surface, $m^2=4q$ and $m-(m-3)$ is squarefree, see \cite{hnr}). Thus we have $a_2\geq 13$ and a jacobian  surface with $a_1=-5$ has at least $q^2+1-5(q+1)+13=5$  points.

If $a_1>-5$, we have
\begin{eqnarray*}
q^2+1+(q+1)a_1+a_2 & \geq & q^2+1(q+1)a_1+2\vert a_1\vert\sqrt{q}-2q\\
& = &  9+5a_1+4\vert a_1\vert \\
& \geq &  9+5\times (-4)+4\times 4 \\
& = & 5
\end{eqnarray*}
(for the third row, note that the function $x\mapsto 5x+4\vert x\vert$ is increasing on $[-4,8]$). Thus an abelian surface with $a_1>-5$ has more than $5$ points and our result is proved.

It remains to prove that $j_4(2)\leq 5$. There exists a jacobian with $(a_1,a_2)=(-5,13)$: this is the jacobian of the quadratic twist of the curve with $(a_1,a_2)=(5,13)$ in the proof of Theorem \ref{Jq2}. This jacobian has $q^2+1-5(q+1)+13=5$  points.

\medskip
\noindent
$\bullet$ If $q=9$ (and $m=6$), using the same argument as in the last step, we must have $a_1 \geq -2m+2$. We look at the rows of Table \ref{tabmin}, beginning by the rows on the top, for which $a_1 = -2m+2$. The first two can be ignored since $\{ 2\sqrt q\} =0$ is less than $\sqrt 3 -1$ and less than $\sqrt 2-1$. An abelian surface of type $[-m,-m+2]$  is not  a jacobian (this is an almost ordinary abelian surface, $m^2=4q$ and $m-(m-2)$ is squarefree, see \cite{hnr}). The product  of two copies of an elliptic curve of trace $(m-1)$  is isogenous to a jacobian (such a curve exists since $3\not\vert (m-1)$).

\bigskip
\noindent
\textbf{b) If $q$ is not a square.}

Using twisting arguments and the proof of Theorem \ref{Jq2}, we see that:

\medskip
\noindent
$\bullet$ There exists a jacobian of type $[-m,-m]$ if and only if $q$ is not special.

\medskip
\noindent
$\bullet$ If $q$ is special, there exists a jacobian of type  $[-m+\frac{1-\sqrt{5}}{2},-m+\frac{1+\sqrt{5}}{2}]$ if and only if  $\{ 2\sqrt q\}\geq\frac{\sqrt 5 -1}{2}$.

\medskip
\noindent
$\bullet$ An abelian surface of type $[-m,-m+1]$ is never a jacobian.

\medskip

Now suppose that $q$ is special and $\{ 2\sqrt q\} <\frac{\sqrt 5 -1}{2}$.

\medskip
\noindent
$\bullet$ In order to have the existence of an abelian surface of type $[-m+1+\sqrt{3},-m+1-\sqrt{3}]$, it is necessary to have $m-1+\sqrt{3}\leq 2\sqrt q$ which is equivalent to $\{ 2\sqrt q\} \geq\sqrt 3 -1$. When $\{ 2\sqrt q\}<\frac{\sqrt 5 -1}{2}$, this condition is never satisfied  (since $\frac{\sqrt 5 -1}{2}<\sqrt 3 -1$).

\medskip
\noindent
$\bullet$ In order to have the existence of an abelian surface of type $[-m+1+\sqrt{2},-m+1-\sqrt{2}]$,   it is necessary to have   $\{ 2\sqrt q\} \geq\sqrt 2 -1$.

Suppose that this condition holds; we will show that there exists an abelian surface of type $[-m+1+\sqrt{2},-m+1-\sqrt{2}]$. We use the same kind of argument as Serre used in \cite{serre2}.
If $p\vert m$, we are done since $p\not\vert a_2=m^2-2m-1+2q$. Otherwise, $(m-2\sqrt q)(m+2\sqrt q)=m^2-4q\in\{ -3,-4,-7\}$, hence $\{ 2\sqrt q\} =2\sqrt q -m=\frac{4q-m^2}{m+2\sqrt q}\leq\frac{7}{2m}$ and if $m\geq 9$, $\frac{7}{2m}<\sqrt{2}-1$. 
It remains to consider by hand the powers of primes of the form $x^2+1$, $x^2+x+1$ and $x^2+x+2$ with $m<9$ (i.e. $q<21$). These powers of primes are precisely 
 $2$, $3$, $4$, $5$, $7$, $8$, $13$ and $17$. 
 For $q=2,8$, we have $\{ 2\sqrt q\}\geq\frac{\sqrt 5 -1}{2}$. For $q=3$, we have $p\vert m$. For $q=4,7,13,17$, we have $\{ 2\sqrt q\} <\sqrt{2}-1$. For $q=5$, $m=4$ and $p=5$ do not divide $a_2=m^2-2m-1+2q=17$. So we are done.

Finally, using \cite{hnr}, we conclude that this abelian surface is isogenous to a jacobian.

\medskip
\noindent
$\bullet$ We suppose that $q$ is special, $\{ 2\sqrt q\} <\sqrt 2-1 $, $p\not\vert m$ and $q\neq 7^3$. Then $p\not\vert (m-2)$.

To see this, we take $p\neq 2$ (if $p=2$, it is obvious) and we use the remark after the definition of "special". Suppose that $p$ divides $(m-2)$, then it also divides $m^2-4-4q=(m+2)(m-2)-4q$. Since $p\neq 2$, we must have $m^2-4q\in\{-3,-4\}$. If $m^2-4q=-3$, $p$ divides $-3-4=-7$ thus $p=7$; $q$ is not prime (since for $q=7$, $p\not\vert (m-2)=5$) therefore we must have $q=7^3$ and this case is excluded. If $m^2-4q=-4$, $p$ divides $-4-4=-8$ thus $p=2$ which contradicts our assumption. Thus the result is proved.

Finally, by \cite{water} there exist elliptic curves of trace $m$ and $(m-2)$ and by \cite{hnr} their product is isogenous to a jacobian.

\medskip
\noindent
$\bullet$ We suppose that $q$ is special, $\{ 2\sqrt q\} <\sqrt 2-1 $ and $p\vert m$, or $q=7^3$. By \cite{water}, if $p\vert m$, there does not exist an elliptic curve of trace $m$ ($q=2$ and $3$ are excluded since in those cases, $\{ 2\sqrt q\} \geq\sqrt 2-1 $). If $q=7^3$ (thus $(m-2)=35$) there does not exist an elliptic curve of trace $(m-2)$. Therefore, in both cases, an abelian surface of type $[-m,-m+2]$ cannot exist. 

\medskip
\noindent
$\bullet$ If $q$ is special, $\{ 2\sqrt q\} <\sqrt 2-1 $ and $p\vert m$ or $q= 7^3$, then there exists a jacobian of type $[-m+1,-m+1]$: this is the jacobian of the quadratic twist of the curve of type $[m-1,m-1]$ in the proof of Theorem \ref{Jq2}.
\end{proof}




\section{Asymptotic parameters for jacobians}

Let us consider the following aymptotic quantities:
$$J_{q}(\infty)= \limsup_{g\rightarrow\infty}(\# J_q(g))^{1/g}\quad \mbox{ and }\quad j_{q}(\infty)= \liminf_{g\rightarrow\infty}(\# J_q(g))^{1/g}.$$

\begin{proposition} We have
$$q\leq j_q(\infty)\leq J_{q}(\infty)\leq q+\sqrt{q}\leq q+2\sqrt{q}+1.$$
\end{proposition}

\begin{proof}
The first inequality comes from the lower bound of Lachaud-Martin-Deschamps, the third one from Proposition \ref{bornetr} and the last one is the Weil bound.
\end{proof}

For $q$ square, we have the following bound proved by Vl\u{a}du\c{t} \cite{vla}.

\begin{proposition} If $q$ is a square,
$$q(\frac{q}{q-1})^{\sqrt{q}-1}\leq J_{q}(\infty).$$
\end{proposition}
However, for $q\gg 0$, we have
$$q(\frac{q}{q-1})^{\sqrt{q}-1}=q+\sqrt{q}-\frac{1}{2}+o(1).$$
Hence, for $q$ large
$$q+\sqrt{q}-\frac{1}{2}+o(1)\leq J_{q}(\infty)\leq q+\sqrt{q},$$ 
which raises to the following question.

\bigskip
\noindent
{\bf Question.}  What are the values between $q+\sqrt{q}-\frac{1}{2}$ and $q+\sqrt{q}$ which are attained by a sequence of curves reaching the Drinfeld-Vl\u{a}du\c{t} bound?


\section*{Acknowledgements}
I would like to thank my thesis advisor, Yves Aubry, for his help and encouragement. I am also grateful to Hamish Ivey-Law for his careful reading of this paper.



\end{document}